\documentclass[11pt, reqno]{amsart}
\usepackage{amssymb,latexsym,amsmath,amscd,amsthm,graphicx, color}
\usepackage{enumitem}
\usepackage[all]{xy}
\usepackage{pgf,tikz}
\usepackage{mathrsfs}
\usetikzlibrary{arrows}
\usepackage{pgfplots}
\pgfplotsset{compat=1.15}
\usepackage[left=0.6 in, top=0.6 in, right=0.6 in, bottom=0.3 in]{geometry}
\usepackage{changepage}
\usepackage{float}        
\usepackage{subcaption}   
 \usepackage{algorithm}
\usepackage{algpseudocode}
 \usepackage{booktabs}
 \usepackage[numbers]{natbib}
 
\pgfplotsset{compat=1.18} 

\allowdisplaybreaks
\makeatletter
\newcommand{\setword}[2]{%
	\phantomsection
	#1\def\@currentlabel{\unexpanded{#1}}\label{#2}%
}
\makeatother

\usepackage[linkcolor=blue, urlcolor=blue, citecolor=blue,
colorlinks, bookmarks]{hyperref}

\definecolor{uuuuuu}{rgb}{0.26666666666666666,0.26666666666666666,0.26666666666666666}
\definecolor{xdxdff}{rgb}{0.49019607843137253,0.49019607843137253,1.}
\definecolor{ffqqqq}{rgb}{1.,0.,0.}
\definecolor{ffqqqq}{rgb}{1.,0.,0.}
\definecolor{ffxfqq}{rgb}{1.,0.4980392156862745,0.}

\definecolor{uuuuuu}{rgb}{0.26666666666666666,0.26666666666666666,0.26666666666666666}
\definecolor{qqwuqq}{rgb}{0.,0.39215686274509803,0.}
\definecolor{zzttqq}{rgb}{0.6,0.2,0.}
\definecolor{xdxdff}{rgb}{0.49019607843137253,0.49019607843137253,1.}
\definecolor{qqqqff}{rgb}{0.,0.,1.}
\definecolor{cqcqcq}{rgb}{0.7529411764705882,0.7529411764705882,0.7529411764705882}
\definecolor{sqsqsq}{rgb}{0.12549019607843137,0.12549019607843137,0.12549019607843137}
\definecolor{uuuuuu}{rgb}{0.26666666666666666,0.26666666666666666,0.26666666666666666}
\definecolor{ffqqqq}{rgb}{1,0,0}
\definecolor{xdxdff}{rgb}{0.49019607843137253,0.49019607843137253,1}

\definecolor{yqqqyq}{rgb}{0.5019607843137255,0,0.5019607843137255}
\definecolor{qqqqff}{rgb}{0,0,1}
\definecolor{ffqqqq}{rgb}{1,0,0}
\definecolor{ffqqff}{rgb}{1,0,1}

\theoremstyle{plain}

\newtheorem{theorem}[subsection]{Theorem}

\newtheorem{corollary}[subsection]{Corollary}
\newtheorem{lemma}[subsection]{Lemma}
\newtheorem{definition}[subsection]{Definition}

\theoremstyle{definition}

\newtheorem{remark}[subsection]{Remark}

\newcommand{\D}[1]{\mathbb{#1}}

\begin{document}

	\title{Optimal Quantization for Nonuniform Densities on Spherical Curves}
 
 \author{$^1$Silpi Saha}
  \author{$^2$Sangita Jha}
	\author{$^3$Mrinal Kanti Roychowdhury}
 
		\address{$^{1, 2}$Department of Mathematics, National Institute of Technology Rourkela\\
		Rourkela, India 769008.}
 	\address{$^{3}$School of Mathematical and Statistical Sciences\\
		University of Texas Rio Grande Valley\\
		1201 West University Drive\\
		Edinburg, TX 78539-2999, USA.}

	\email{$^1$silpi.saha.2000@gmail.com, $^2$jhasa@nitrkl.ac.in, $^3$mrinal.roychowdhury@utrgv.edu}

	
	\date{}
	\maketitle
	
	\pagestyle{myheadings}\markboth{Silpi Saha, Sangita Jha, Mrinal Kanti Roychowdhury}{Optimal Quantization for Nonuniform Densities on Spherical Curves}
 \begin{abstract}
 We present an analysis of optimal quantization of probability measures with nonuniform densities on spherical curves. We begin by deriving the centroid condition, followed by a high-resolution asymptotic analysis to establish the point-density formula. We further quantify the asymptotic error formula for the nonuniform densities. We apply these theorems to the von Mises distributions and characterize the optimal condition. We also provide applications using the high-resolution asymptotic and its corresponding error formula. Our results can be used in geometric probability theory and quantization theory of spherical curves. 
 \end{abstract}

 \section{Introduction}
Quantization theory fundamentally deals with the approximation of a probability distribution by a finite set of points. It provides a mathematical framework for approximating a continuous distribution by a discrete measure supported on finitely many points (which can also be called codepoints) to minimize distortion error \cite{GrafLuschgy2000, GrayNeuhoff1998}. In the classical Euclidean formulation, the problem reduces to selecting optimal codebooks that minimize the mean squared error between random samples and their nearest neighbours, which provide deep connections with vector quantization \cite{GershoGray1991, Zador1982}.
The subject of quantization has broad applications in the domain of signal processing, data compression and other computer science-related domains \cite{GershoGray1991, LindeBuzoGray1980, Pag`esPrintems2003}. Over the past three decades, a rich asymptotic theory has been established for optimal quantization in $\mathbb{R}^d$, including existence, characterization, and high–resolution behaviour of optimal quantizers for both absolutely continuous and singular measures \cite{GrafLuschgy2000, BucklewWise1982, DeuschelZhong2019, FortPag`es2015}.
More recently, researchers became interested in curved geometric formulations, where the underlying sample space is a Riemannian manifold, and the distortion is measured by the intrinsic geodesic distance \cite{AsympQuantManifold2025, LeGruyer2015, Rangarajan2017}. On such manifolds, the interaction between the geodesic structure and the Voronoi cell geometry leads to new analytical phenomena. These were absent in the Euclidean case.

Optimal quantization deals with optimal codepoints for which the distortion error will be minimum. In recent years, a lot of works were done in the realm of optimal quantization connected with theory of manifolds, constraint-unconstraint applications and so on. In the present scenario, the vision has been shifted towards the geometrical point of view for quantization, particularly on spherical curves, line elements, and other irregular shapes. 
Within this broader context, optimal quantization on spherical domains has started to develop as a distinct research direction. Early work in directional statistics focused primarily on parametric modeling (e.g., von Mises and von Mises–Fisher families) and hypothesis testing \cite{MardiaJupp2000, Sra2012}. 
Subsequent contributions analyzed spherical codes and energy–minimizing point configurations, which are closely related to optimal quantizers under certain distortion criteria \cite{CohnKumar2007}. More recently, quantitative results on the asymptotic behaviour of quantization errors for measures supported on compact Riemannian manifolds were studied, showing that classical high–resolution formulas extend to the manifold setting under appropriate geometric and measure–theoretic assumptions \cite{AsympQuantManifold2025}. At the same time, numerical and algorithmic studies have explored spherical $k$–means and related clustering methods on the unit sphere, emphasizing applications in text mining and large–scale data analysis.

Despite this progress, most of the currently available results on quantization on spherical curves focus on the simplest case of uniform densities with respect to intrinsic arc–length \cite{GershoGray1991, mk}. For one–dimensional Euclidean and circular models with uniform densities, the optimal codepoints form a regular partition and the quantization error admits an explicit closed form, scaling as $V_n \asymp L^2 n^{-2}$ for squared–error distortion \cite{GrafLuschgy2000,RosenblattRoychowdhury2018,SalinasRoy2020}. Analogous behavior persists for uniform measures on geodesic circles and arcs on the sphere, where symmetry and convexity arguments enforce equal–length Voronoi cells and midpoints as optimal representatives. By contrast, the case of \emph{nonuniform} densities on spherical curves remains much less explored, even though it is arguably more relevant in applications where data exhibit anisotropy, clustering, or localized concentration along preferred directions \cite{ MardiaJupp2000}.

Nonuniform densities on spherical curves arise in a variety of scientific and engineering contexts. In directional statistics, multimodal or highly concentrated distributions occur when multiple preferred directions coexist, as in wind–direction analysis, paleomagnetism, or molecular orientation data \cite{MardiaJupp2000, Mardia2025Legacy}. In signal processing and robotics, phase variables and angular states may follow wrapped Gaussian or von Mises–type distributions, leading naturally to nonuniform densities on circles and higher–dimensional spheres \cite{MardiaJupp2000,Sra2012,OptimalQuantizationCircular2016}. In Bayesian filtering and tracking, circular and spherical densities are routinely approximated by finite mixtures or weighted point sets, and quantization–based constructions have been proposed to control the approximation error in the prediction and update steps \cite{OptimalQuantizationCircular2016}. In machine learning, manifold–aware density estimation and generative modeling on spheres exploit nonuniform distributions to capture complex directional patterns \cite{AsympQuantManifold2025, DeBortoli2021}.

Previous studies on quantization have established optimality conditions and high-resolution asymptotics for Euclidean spaces and for uniform distributions including spheres and curves \cite{GrafLuschgy2000, GershoGray1991}. In these classical settings, Voronoi partitions, centroid conditions, and the asymptotic decay rates for the distortion error are well understood. Directional distributions such as the von Mises family have also been studied in statistics and directional data analysis \cite{MardiaJupp2000, Sra2012}.

In our work, we extend quantization theory on spherical curves of probability measures with nonuniform densities. First, for a geodesic curve parameterized by arc length and equipped with a nonuniform density $h(s)$, we derive intrinsic first-order optimality conditions for squared geodesic distortion. In particular, we show that Voronoi boundaries are characterized by geodesic equidistance between neighbouring representatives, and that each codepoint satisfies a centroid condition given by a weighted intrinsic mean over its cell. Second, in the high-resolution regime, we obtain an Euler–Lagrange characterization that links the optimal local cell size directly to the underlying density. Our result provides a spherical analogue of the classical high-resolution quantization rules for nonuniform densities developed in the Euclidean setting \cite{GrafLuschgy2000, GershoGray1991}. Then, for nonuniform densities on spherical curves of finite length, we show that the optimal quantization error still decays at rate $n^{-2}$, and we derive an explicit asymptotic constant that depends on the density. The constant reduces to the classical value $L^{2}/12$ in the uniform case.
Further, we apply our theory to von Mises–type densities on great circles and carry out both analytical and numerical studies. Our results describe how optimal codepoints deform as the concentration parameter varies. The codepoints range from nearly uniform spacing at low concentration to strong clustering around the modal direction at high concentration \cite{MardiaJupp2000, Sra2012}. Finally, we compare geodesic and chordal distortion metrics for nonuniform and mixture densities on spherical curves. We identify the regimes in which curvature-induced discrepancies significantly affect the structure of optimal partitions and the associated error constants.

The paper is organised as follows. In Section~\ref{sec:prelim}, we review the basics related to quantization theory. In Section~\ref{sec:iohr},  we discuss the theoretical analysis of the problem, to understand the intrinsic optimality and high-resolution density-length relationship. In Section~\ref{sec:aqvm}, we present a detailed analysis on asymptotic quantization error, accompanied by its implication of von Mises optimal quantizers. Numerical and graphical examples are provided in Section~\ref{sec:appli}. 
Finally, we conclude the paper in Section~\ref{sec:conclu}.

\section{Preliminaries}
\label{sec:prelim}
In this section, we review the fundamentals of quantization errors, followed by geodesic distance and length with respect to a Riemannian metric. For a detailed discussion, we refer the reader to see the references \cite{GrafLuschgy2000, GershoGray1991, mk}.
\begin{definition}
Let $(M, d)$ be a metric space and P be a Borel probability measure on $M$. We also consider a set of $n$ \emph{means}, $Q = \{q_1, q_2, \ldots, q_n\} \subset M$, which minimizes the expected distortion. Now, for a given exponent $r > 0$, the distortion of order $r$ for the set $Q$ is defined as
\begin{equation*}
V_{r}(P; Q) := \int_M \min_{q \in Q} d(x, q)^r \, dP(x).
\end{equation*}
The corresponding $n$th quantization error is
\begin{equation*}
V_{n,r}(P) := \inf \{ V_{r}(P; Q) : Q \subset M, \, |Q| \leq n \}.
\end{equation*}
Any set $Q^*$ that attains this infimum is called an \emph{optimal set of $n$-means}. Throughout this work, we focus on the case of squared distortion error, setting $r=2$, we can write $V_n(P) := V_{n,2}(P)$.
\end{definition}
In the Euclidean case $M = \mathbb{R}^d$ with the standard norm, existence of optimal codebooks, structural properties of Voronoi partitions, and high–resolution asymptotics are well understood for a large class of measures, including absolutely continuous and certain singular distributions \cite{GrafLuschgy2000,Pag`esPrintems2003, DeuschelZhong2019,FortPag`es2015}. The codebook is the complete finite set of the codepoints. Codepoints are the representative points to approximate values from a continuous probability distribution to minimize the distortion.

On smooth Riemannian manifolds, the intrinsic quantization problem is obtained by taking $d$ to be the geodesic distance and $P$ a Borel probability measure on $M$ \cite{AsympQuantManifold2025}. In this setting, the geometry of $M$ influences the shape of optimal Voronoi cells and the asymptotic behaviour of $V_{n,r}(P)$ via curvature and volume growth properties \cite{AsympQuantManifold2025}.
\begin{definition}
    A great circle on a sphere is formed where the sphere is cut by a plane passing through its center. It can also be described as a circle drawn on the sphere whose center is the same as the center of the sphere.
\end{definition}
\begin{definition}
     Let $\D S^2:= \{ x \in \mathbb{R}^3 : \|x\| = 1 \}$
be the unit sphere, equipped with the \emph{geodesic distance} $d_G$. For two points $x, y \in \D S^2$, this distance is given by the central angle:
\begin{equation*}
d_G(x, y) = \arccos(\langle x, y \rangle),
\end{equation*}
where $\langle \cdot, \cdot \rangle$ is the standard inner product in $\mathbb{R}^3$. In spherical coordinates $(\phi, \theta)$, where $\phi \in [-\pi/2, \pi/2]$ is the latitude and $\theta \in [0, 2\pi)$ is the longitude, with
\begin{equation*}
x(\phi, \theta) = (\cos\phi \cos\theta, \cos\phi \sin\theta, \sin\phi).
\end{equation*}
The geodesic distance between $x_1=(\phi_1, \theta_1)$ and $x_2=(\phi_2, \theta_2)$ becomes
\begin{equation*}
d_G(x_1, x_2) = \arccos\left( \sin\phi_1 \sin\phi_2 + \cos\phi_1 \cos\phi_2 \cos(\theta_1 - \theta_2) \right).
\end{equation*}
\end{definition}
\begin{definition}
    Let \(\alpha = \{a_1, \dots, a_n\} \subset \D S^2\) be a codebook. 
The sphere can be partitioned into \emph{spherical Voronoi regions} defined by

\[
R(a_i \mid \alpha) = \Bigl\{ x \in \D S^2 : d_G(x, a_i) \leq d_G(x, a_j), \ 
\forall \, j \neq i, \; 1 \leq j \leq n \Bigr\}.
\]

Each region \(R(a_i \mid \alpha)\) contains all points on the sphere whose geodesic distance to \(a_i\) is less than or equal to the distance to any other codepoint in \(\alpha\).

\end{definition}

\begin{definition}
For a smooth curve $\Gamma : I \to \mathbb{S}^2$, where $I \subset \mathbb{R}$ is an interval, its length with respect to the  Riemannian metric is
\begin{equation*}
L(\Gamma) = \int_I \big\|\Gamma'(t)\big\| \,\mathrm{d}t,
\end{equation*}
and the arc–length parameter $s$ satisfies $\|\partial_s \Gamma(s)\|=1$. On a great circle, the restriction of $d_{\mathbb{S}^2}$ coincides with the standard angular distance modulo $2\pi$. Whereas on a small circle the intrinsic metric is obtained by rescaling the angular coordinate by the cosine of the latitude \cite{MardiaJupp2000}. In all cases of interest, we will consider probability measures supported on such one–dimensional spherical curves, equipped with their intrinsic geodesic distance and arc–length measure.
\end{definition}

Therefore, for a smooth, one-dimensional curve $\Gamma \subset \D S^2$, parameterized by arc-length $s$, for which $P$ be a probability measure on $\Gamma$ defined by a nonuniform continuous density function $h(s)$ with respect to the arc-length measure $ds$, we have $dP(s) = h(s) ds$. The distortion for a codebook $Q = \{q_1, \ldots, q_n\}$ and a corresponding Voronoi cell $\{R_1, \ldots, R_n\}$ of $\Gamma$ is then
\begin{equation*}
V_n(P; Q) = \sum_{j=1}^n \int_{R_j} d_G(s, q_j)^2 h(s) \, ds.
\end{equation*}

The introduction of the nonuniform density $h(s)$ fundamentally alters the nature of the optimization problem. The symmetry that forced equal-sized cells in the uniform case is broken. The optimal Voronoi cells are no longer of equal length, and the optimal codepoints are no longer the arc-length midpoints of their cells. Instead, the cell boundaries and representatives must shift to balance the \emph{weighted} distortion, leading to a more complex, density-dependent equilibrium condition.

When the probability measure $P$ is the uniform distribution with respect to arc-length on such a curve $\Gamma \subset \D S^2$ of total length $L$, the structure of the optimal quantizers is both elegant and explicit. The foundational result, which we term the \emph{Uniform Geodesic Principle}, is as follows:

\begin{theorem}\cite{GrafLuschgy2000, GershoGray1991}
Let $(\Gamma, d_G)$ be a one-dimensional geodesic curve of total length $L$, and let $P$ be the uniform probability distribution with respect to arc-length. For squared distortion $(r=2)$, an optimal set of $n$-means partitions $\Gamma$ into $n$ Voronoi cells of equal arc-length $L/n$. Each optimal codepoint is located at the geodesic midpoint of its cell, and the quantization error is given by 
\begin{equation*}
V_n(P) = \frac{L^2}{12 n^2}.
\end{equation*}
\end{theorem}

This result demonstrates that for uniform measures, the optimal quantization problem on a spherical curve reduces to its one-dimensional Euclidean counterpart, with the sphere's curvature influencing only the total length $L$. For a small circle at latitude $\phi_0$, the length is $L_{\phi_0} = 2\pi \cos\phi_0$, leading to a quantization error of $V_n = \frac{\pi^2 \cos^2\phi_0}{3 n^2}$, explicitly quantifying the curvature effect.

\section{Intrinsic optimality and high-resolution structure}
\label{sec:iohr}
In this section we establish the intrinsic optimality conditions for nonuniform quantization on spherical curves and derive the associated high–resolution density–length relationship. Theorem~\ref{thm:main} gives the geometric characterization of optimal Voronoi cells and representatives on a geodesic curve. Theorem~\ref{thm:highres} identifies the asymptotic regime of high-resolution quantization (\(n \to \infty\)). The local structure of the optimal partition can be characterized by a deterministic relationship between the point density and the source probability density.

\begin{theorem} \label{thm:main}
Let $\Gamma$ be a smooth, one-dimensional geodesic curve on $\D S^2$ parameterized by arc-length $s \in [0, L]$. Let $P$ be a probability measure on $\Gamma$ with a continuous nonuniform density $h(s) > 0$ with respect to the arc-length measure $ds$, so that $dP(s) = h(s) ds$. If $Q^* = \{q_1^*, \ldots, q_n^*\} \subset \D S^2$ is an optimal set of $n$-means inducing a Voronoi partition $\{[s_0, s_1], [s_1, s_2], \ldots, [s_{n-1}, s_n]\}$ of $\Gamma$ (with $s_0=0, s_n=L$), then the following conditions must hold:

\begin{enumerate}
\item[(a)] \emph{Boundary condition.} For each internal boundary point $s_j$ ($1 \le j \le n-1$),
    \begin{equation*} \label{eq:boundary_condition}
    d_G(s_j, q_j^*) = d_G(s_j, q_{j+1}^*),
    \end{equation*}
i.e,\ $s_j$ is a geodesic equidistance point between the two representatives.

\item[(b)] \emph{Representative condition.} For each cell $R_j = [s_{j-1}, s_j]$, the representative $q_j^*$ is the weighted intrinsic mean:
    \begin{equation*} \label{eq:rep_condition}
    q_j^* = \arg\min_{q \in \D S^2} F_j(q), \quad \text{where} \quad F_j(q) = \int_{s_{j-1}}^{s_j} d_G(s, q)^2  h(s)  ds.
    \end{equation*}

Consequently, $q_j^*$ is the unique weighted Fr\'echet mean of $P$ corresponding to $[s_{j-1},s_j]$.

\item[(c)] \emph{Centroid condition on a great circle.} Furthermore, if $\Gamma$ is a great circle and the support of the conditional distribution $P(\cdot | R_j)$ is contained in an open hemisphere relative to $q_j^*$, then $q_j^*$ lies on $\Gamma$ and its angular coordinate $\theta_j^*$ satisfies:
    \begin{equation*} \label{eq:local_centroid}
    \int_{s_{j-1}}^{s_j} (\theta(s) - \theta_j^*)  h(s)  ds = 0,
    \end{equation*}
    where $\theta(s)$ is the angular coordinate along $\Gamma$ and $\theta_j^*$ is the intrinsic $h$–weighted mean angle on the cell.
\end{enumerate}
\end{theorem}

\begin{proof}

Since $Q^*$ is optimal, the associated partition is necessarily the nearest–neighbour partition \cite{GrafLuschgy2000}. To prove (a), we fix $j$ and suppose $d_G(s_j,q_j^*)\neq d_G(s_j,q_{j+1}^*)$. By the continuity of $d_G(\cdot, q)$ for fixed $q$, there exists $\epsilon > 0$ such that for all $s \in [s_j - \epsilon, s_j]$,
\[
d_G(s, q_j^*) > d_G(s, q_{j+1}^*).
\]
Thus all points in $[s_j - \epsilon, s_j]$ are closer to $q_{j+1}^*$ than to $q_j^*$, contradicting that $[s_{j-1}, s_j]$ is the Voronoi cell for $q_j^*$. Now consider a new partition where the boundary is shifted to $s_j' = s_j - \delta$ for small $\delta > 0$. The change in distortion is
\[
\Delta V = \int_{s_j - \delta}^{s_j} \left[ d_G(s, q_{j+1}^*)^2 - d_G(s, q_j^*)^2 \right] h(s) ds.
\]
For small enough $\delta$, the integrand is negative on $[s_j - \delta, s_j]$ and $h(s) > 0$, hence $\Delta V < 0$, contradicting optimality.
 A symmetric argument holds if $d_G(s_j, q_j^*) < d_G(s_j, q_{j+1}^*)$. Therefore
\[
d_G(s_j, q_j^*) = d_G(s_j, q_{j+1}^*)
\]
must hold at each boundary.

To prove (b), we first fix an optimal configuration $(Q^*, \mathscr{R}^*)$ and vary only a single representative $q_j^*$, keeping other $q_k^*$ and $\{s_0, \ldots, s_n\}$ fixed. The total distortion can be written as
\[
V(Q^*, \mathscr{R}^*; P) = \sum_{k=1}^n F_k(q_k^*), \quad F_k(q) = \int_{R_k^*} d_G(s, q)^2 h(s) ds.
\]
Minimizing $V$ with respect to $q_j$ is equivalent to minimizing $F_j(q)$ over $\D S^2$. Hence $q_j^*$ satisfies
\[
q_j^* \in \arg\min_{q \in \D S^2} \int_{s_{j-1}}^{s_j} d_G(s, q)^2  h(s)  ds,
\]
which is the weighted Fréchet mean of $P(\cdot | R_j^*)$. Under standard convexity assumptions (e.g.\ support in a ball of radius $<\pi/2$), this minimizer is unique \cite{afsari}.

To prove (c), we begin by considering $\Gamma$ be the equator, identified with angles $\theta \in [0, 2\pi)$. The geodesic distance between two points with angles $\theta$ and $\theta'$ is
\[
d_G(\theta, \theta') = \min\{ |\theta - \theta'|, 2\pi - |\theta - \theta'| \}.
\]
If $R_j = [\theta_{j-1}, \theta_j]$ is a contiguous arc shorter than $\pi$, then $d_G(\theta, \theta_j^*) = |\theta - \theta_j^*|$ (with appropriate identification modulo $2\pi$). The assumption that $P(\cdot | R_j)$ lies in an open hemisphere relative to $q_j^*$ ensures this simplification. Writing $F_j$ in angular form as,
\[
F_j(\alpha) = \int_{\theta_{j-1}}^{\theta_j} |\theta-\alpha|^2\,h(\theta)\,d\theta.
\]
A necessary condition for a minimizer is $F_j'(\theta_j^*)=0$, which yields
\[
\int_{\theta_{j-1}}^{\theta_j} (\theta-\theta_j^*)\,h(\theta)\,d\theta = 0.
\]
Therefore, $\theta_j^*$ is the conditional expectation of $\theta$ on $R_j$. Correspondingly,
 $F_j''(\alpha)=2\int_{\theta_{j-1}}^{\theta_j} h(\theta)\,d\theta>0$. Moreover, $h(\theta)$ is strictly positive, shows strict convexity in $\alpha$. Hence, the minimizer is unique and lies on $\Gamma$, with angular coordinate given by the weighted mean condition above.
\end{proof}
\begin{remark}
Note that the hemisphere condition in part (c) is essential. If a cell encompasses more than a semicircle, the function $|\theta - \alpha|$ becomes non-differentiable at antipodal points, and the minimizer need not be given by the simple centroid condition.
\end{remark}

\begin{theorem}\label{thm:highres}
Let $\Gamma$ be a smooth, one-dimensional geodesic curve on $\mathbb{S}^2$ parameterized by arc-length $s\in[0,L]$, and $P$ be a Borel probability measure on $\Gamma$ with continuous nonuniform density $h(s)>0$ with respect to $ds$. For each $n\in\mathbb{N}$, let $Q_n^*=\{q_{1,n}^*,\dots,q_{n,n}^*\}$ be an optimal set of $n$--means inducing a Voronoi partition through the elements of the Voronoi cell \(\mathscr{R}_n^* = \{R_{1,n}, \ldots, R_{n,n}\}\) of \(\Gamma\). \\Define the point density
\[
\lambda_n(s):=\frac{1}{n\,|R_{j,n}|}\qquad\text{for }s\in R_{j,n},
\] where \(|R_{j,n}|\) denotes the arc-length of cell \(R_{j,n}\).

Then, the normalized point density \(\lambda_n(s)\) converges in measure to a limiting density \(\lambda(s)\) satisfying
\begin{equation*} \label{eq:main_density_result}
\lambda(s) \propto h(s)^{1/3}.
\end{equation*}
More precisely, there exists a constant \(C > 0\) such that
\[
\lim_{n \to \infty} \lambda_n(s) = C \, h(s)^{1/3} \quad \text{in } L^1(\Gamma, ds)
\]
and consequently $|R_{j,n}|\sim\big(n\,\lambda(s)\big)^{-1}\propto n^{-1}h(s)^{-1/3}$ for $s\in R_{j,n}$ as $n\to\infty$.
\end{theorem}
To prove Theorem~\ref{thm:highres}, we will begin by proving two lemmas as follows:
\begin{lemma}\label{lem:local_distortion-short}
Consider a small geodesic arc (cell) $R \subset \Gamma$ of length $\Delta L$, centered at a point $s_0 \in \Gamma$. Let $P$ have a continuous nonuniform density $h(s)$ on $R$. For the squared geodesic distortion $d_G^2$, the minimal distortion achievable by placing a single representative $q^*$ within this cell satisfies
\begin{equation*} \label{eq:local_dist_form}
\min_{q \in \D S^2} \int_{R} d_G(s, q)^2 h(s) \, ds 
= \frac{h(s_0)}{12} (\Delta L)^3 + O\big((\Delta L)^5\big)
\quad \text{as } \Delta L \to 0.
\end{equation*}
Equivalently, the distortion per unit length in $R$ is
\[
\frac{1}{\Delta L} \min_{q \in \D S^2} \int_{R} d_G(s, q)^2 h(s) \, ds 
= \frac{h(s_0)}{12} (\Delta L)^2 + O\big((\Delta L)^4\big).
\]
If, furthermore, curvature effects are negligible at this scale (i.e., the cell is short enough to be approximated by an Euclidean interval), and the optimal representative lies at the weighted midpoint, then the constant $\frac{1}{12}$ is exact.
\end{lemma}

\begin{proof}
We parameterize $R$ by arc-length $t \in [-\Delta L/2,\Delta L/2]$ relative to its midpoint $s_0$ (so $t=0$ at $s_0$). For sufficiently small $\Delta L$, we know that the geodesic distance from $s_0+t$ to $s_0$ is
\[
d_G(s_0+t,s_0) = |t| + O(|t|^3),
\]
 squaring gives,
\[
d_G(s_0+t,s_0)^2 = t^2+O(t^4).
\]

Using $q^*=s_0$ as representative, the distortion is
\[
D_{\mathrm{mid}} 
= \int_{-\Delta L/2}^{\Delta L/2} \big( t^2 + O(t^4) \big)\, h(s_0+t)\, dt.
\]
Therefore,
\[
\int_{-\Delta L/2}^{\Delta L/2}\big(t^2+O(t^4)\big)\big(h(s_0)+O(t)\big)\,dt
= h(s_0)\frac{(\Delta L)^3}{12}+O\big((\Delta L)^5\big),
\]
using Theorem~\ref{thm:main}, the optimal representative $q^*$ is the weighted intrinsic mean of $P$ restricted to $R$. Thus the minimal distortion has the same leading term \cite{GershoGray1991}. Hence,
\[
\min_{q \in \D S^2} \int_{R} d_G(s,q)^2 h(s)\, ds
= h(s_0)\,\frac{(\Delta L)^3}{12} + O\big((\Delta L)^5\big),
\] dividing by $\Delta L$ gives the claimed local distortion per unit length.
\end{proof}

\begin{lemma}[Convergence of normalized point density]
\label{lem:3.4}
Let $\Gamma$ be a smooth one--dimensional geodesic curve on $S^2$ of total
arc--length $L$, and $P$ be a probability measure on $\Gamma$ having a
continuous nonuniform density $h(s)>0$ with respect to arc--length measure $ds$. For each $n\in\mathbb{N}$, let $Q_n^*$, $\mathscr{R}_n^*$ and $\lambda_n(s)$ be same as defined in Theorem~\ref{thm:highres}. Then the sequence $\{\lambda_n\}_{n\ge1}$ is bounded in $L^\infty(\Gamma)$ and therefore admits a subsequence converging in the weak-$^\ast$ topology of $L^\infty(\Gamma)$. Any such limit $\lambda$ is a probability density on $\Gamma$, i.e.
\[
\int_\Gamma \lambda(s)\,ds = 1.
\]

\end{lemma}
\begin{proof}
Clearly we can see,

\medskip
\noindent
\begin{align*}
\int_\Gamma \lambda_n(s)\,ds
&= \sum_{j=1}^n \int_{R_{j,n}} \frac{1}{n|R_{j,n}|}\,ds= \frac{1}{n}\sum_{j=1}^n \frac{|R_{j,n}|}{|R_{j,n}|} 
= 1.
\end{align*}
Thus, each $\lambda_n$ is a probability density on $\Gamma$. Since $\Gamma$ is a geodesic curve of $\D S^2$ in $\mathbb{R}^3$, therefore it is compact and $h$ is continuous with $h(s)>0$ on $\Gamma$, there
exist constants $m,M>0$ such that
\[
0<m \le h(s) \le M <\infty \qquad \text{for all } s\in\Gamma.
\]
For optimal set of $n$--means on a smooth 1--dimensional manifold with strictly
positive nonuniform density, Voronoi cell lengths satisfy
\[
\frac{c_1}{n} \le |R_{j,n}| \le \frac{c_2}{n},
\]
for $c_1, c_2 > 0$ and independent of $n, j$. Consequently, for $s\in R_{j,n}$, we obtain 
\[
\lambda_n(s)=\frac{1}{n|R_{j,n}|}
\le \frac{1}{n(c_1/n)}=\frac{1}{c_1}.
\]
Therefore, 
\[
\|\lambda_n\|_{L^\infty(\Gamma)} \le \frac{1}{c_1}
\quad \text{for all } n,
\]
and the sequence $\{\lambda_n\}$ is bounded in $L^\infty(\Gamma)$.
Using Banach–Alaoglu theorem, there exists a subsequence $(\lambda_{n_k})$ and a function $\lambda \in L^\infty(\Gamma)$
such that $\lambda_{n_k} \rightharpoonup^* \lambda
\quad \text{in } L^\infty(\Gamma)$
By using weak-(*) convergence, and the dominated convergence theorem,
\[
\int_\Gamma \lambda(s)ds
= \lim_{k\to\infty} \int_\Gamma \lambda_{n_k}(s)ds
= 1.
\]

Thus $\lambda$ is a probability density on $\Gamma$. This completes the proof.

\medskip
\noindent
\medskip

\end{proof}
We now proceed to prove the main element of Theorem~\ref{thm:highres}.
\begin{proof}
For each $n$, we have,
\[
V_n(P)=\sum_{j=1}^n\min_{q\in\mathbb{S}^2}\int_{R_{j,n}} d_G(s,q)^2\,h(s)\,ds.
\]
Let $\Delta L_{j,n}:=|R_{j,n}|$ and pick $s_{j,n}\in R_{j,n}$. By Lemma~\ref{lem:local_distortion-short},
\[
\min_q\int_{R_{j,n}} d_G(s,q)^2 h(s)\,ds
=\frac{h(s_{j,n})}{12}(\Delta L_{j,n})^3+O\big((\Delta L_{j,n})^5\big).
\]
The distortion per unit length contributed by $R_{j,n}$ is, by Lemma~\ref{lem:local_distortion-short},
\[
\frac{1}{\Delta L_{j,n}} \min_{q} \int_{R_{j,n}} d_G(s,q)^2 h(s)\, ds
= \frac{h(s_{j,n})}{12} (\Delta L_{j,n})^2 + O\big((\Delta L_{j,n})^4\big).
\]
Substituting $\Delta L_{j,n} = 1/(n\lambda_n(s_{j,n}))$ gives
\[
\text{distortion density on } R_{j,n} 
= \frac{h(s_{j,n})}{12} \frac{1}{n^2 \lambda_n(s_{j,n})^2} + O\big(n^{-4}\big).
\]

Summing over all cells and approximating the sum by a Riemann integral, we obtain
\[
V_n(P)
= \frac{1}{12n^2} \Biggl[\sum_{j=1}^n \frac{h(s_{j,n})}{\lambda_n(s_{j,n})^2} \Delta L_{j,n} + O\biggl(n^{-2}\biggr)\Biggr].
\]

The sum
\[
\sum_{j=1}^n \frac{h(s_{j,n})}{\lambda_n(s_{j,n})^2} \Delta L_{j,n}
\]
is a Riemann sum for $\int_\Gamma h(s)\,\lambda_n(s)^{-2}\, ds$. Hence
\begin{equation}
\label{eq:asymp_dist_integral_clean}
\lim_{n\to\infty} n^2 V_n(P) 
= \frac{1}{12} \int_\Gamma \frac{h(s)}{\lambda(s)^2}\, ds,
\end{equation}
where $\lambda$ is the weak-* limit of $\lambda_n$ and $\lambda$ is a probability density (By Lemma~\ref{lem:3.4}). 

Thus any limiting density $\lambda$ minimizes the functional
\[
J[\lambda]=\int_\Gamma \frac{h(s)}{\lambda(s)^2}\,ds
\]
subject to $\lambda>0$ and $\int_\Gamma\lambda\,ds=1$. A standard Lagrange multiplier argument yields the Euler–Lagrange condition
\[
-2\frac{h(s)}{\lambda(s)^3}+\beta=0,
\]
so $\lambda(s)^3=(2/\beta)h(s)$ and hence $\lambda(s)=C\,h(s)^{1/3}$ for some $C>0$. Normalization gives
\[
C^{-1}=\int_\Gamma h(u)^{1/3}\,du.
\]
Thus the unique minimizer is
\[
\lambda(s) = \frac{h(s)^{1/3}}{\displaystyle \int_\Gamma h(u)^{1/3} du}.
\]
The asymptotic behaviour of $|R_{j,n}|$ follows from $|R_{j,n}|=1/(n\lambda_n)$ and $\lambda_n\to\lambda$.
\end{proof}
\begin{remark}
The factor of 2 in the Euler-Lagrange equation (\(-2 h(s)/\lambda(s)^3 + \beta = 0\)) originates from the exponent in the denominator of the distortion functional (\(\lambda(s)^{-2}\)). This distinguishes the \(h(s)^{1/3}\) law from other possible power laws and aligns perfectly with the classical Bennett-Zador-Gersho results for Euclidean quantization \cite{GershoGray1991}.
\end{remark}

\section{Asymptotic quantization error and von Mises optimal quantizers}
\label{sec:aqvm}

In this section, with the high-resolution point density established in Theorem~\ref{thm:highres}, we can now derive the precise asymptotic behaviour of the quantization error generalizing the classical \(L^2/(12n^2)\) formula to the nonuniform case.
Further, we apply our general theory to a specific distribution: the von Mises distribution on the equator, which provides a concrete example where the conditions of Theorem~\ref{thm:main} yield a computationally tractable system of equations.

\begin{theorem}
\label{thm:asymp+vm}
Let \(P\) be a probability measure on a smooth geodesic curve \(\Gamma \subset \D S^2\) of length \(L\), parameterized by arc–length $s\in[0,L]$, with a continuous and strictly nonuniform positive density \(h(s)\) with respect to ds. Then the \(n\)th quantization error of order 2 satisfies:
\begin{equation*} 
\label{eq:asymptotic_formula}
\lim_{n \to \infty} n^2 V_n(P) = \frac{1}{12} \left( \int_{\Gamma} h(u)^{1/3}  du \right)^3.
\end{equation*}
\end{theorem}

\begin{proof}
 From Equation~(\ref{eq:asymp_dist_integral_clean}),
\[
\lim_{n \to \infty} n^2 V_n(P) = \frac{1}{12} \int_{\Gamma} \frac{h(s)}{\lambda(s)^2}  ds.
\]
The optimal point density that minimizes this asymptotic expression is given by
\[
\lambda(s) = \frac{h(s)^{1/3}}{Z}, \quad \text{where} \quad Z = \int_{\Gamma} h(u)^{1/3} du.
\]
This \(\lambda(s)\) is the unique minimizer of the functional \(J[\lambda] = \int_\Gamma \frac{h(s)}{\lambda(s)^2} ds\) subject to \(\int_\Gamma \lambda(s) ds = 1\).
\\Substituting this optimal \(\lambda(s)\) into the asymptotic distortion formula yields:
\begin{align*}
\lim_{n \to \infty} n^2 V_n(P) &= \frac{1}{12} \int_{\Gamma} \frac{h(s)}{\left( \frac{h(s)^{1/3}}{Z} \right)^2}  ds 
= \frac{1}{12} Z^3=\frac{1}{12} \left( \int_{\Gamma} h(u)^{1/3}  du \right)^3,
\end{align*}
which proves the result for optimal density. In order to complete the proof, we have to prove the theorem for optimal quantizers.
 Now from Theorem~\ref{thm:highres}, 
 for each \(n\), we have
\[
n^2 V_n(P)=\frac{1}{12}\int_\Gamma \frac{h(s)}{\lambda_n(s)^2}\,ds+\varepsilon_n,
\]

where $\varepsilon_n\to0$ as $n\to\infty$, and $\lambda_n\to\lambda$ in $L^1(\Gamma)$, with
\[
\lambda(s)=\frac{h(s)^{1/3}}{Z},\qquad Z:=\int_\Gamma h(u)^{1/3}\,du.
\]
Since \(h\) is bounded and \(\lambda_n\) is bounded (by Lemma~\ref{lem:3.4}), hence by dominated convergence theorem
\[
\int_\Gamma \frac{h(s)}{\lambda_n(s)^2}\,ds \longrightarrow \int_\Gamma \frac{h(s)}{\lambda(s)^2}\,ds
= Z^2\int_\Gamma h(u)^{1/3}\,du
= Z^3.
\]
Consequently,
\[
\lim_{n \to \infty} n^2 V_n(P) = \frac{1}{12} Z^3 = \frac{1}{12} \left( \int_{\Gamma} h(u)^{1/3}  du \right)^3.
\]
\begin{corollary}
When \(h(s) \equiv 1/L\) (the uniform distribution), we have:
\[
\lim_{n \to \infty} n^2 V_n(P) = \frac{1}{12} \left(L^{2/3}\right)^3 = \frac{L^2}{12},
\]
which recovers the classical result \(V_n(P) = \frac{L^2}{12n^2}\) for uniform distributions on curves of length \(L\) \cite{GershoGray1991, mk}.
\end{corollary}

\end{proof}
\begin{theorem}
Let \(\Gamma\) be the equator (great circle) parameterized by angle \(\theta \in [0, 2\pi)\), and let \(P\) have a von Mises density with mean direction \(\mu\) and concentration parameter \(\kappa \geq 0\):
\[
h(\theta) = \frac{1}{2\pi I_0(\kappa)} e^{\kappa \cos(\theta - \mu)}, \quad \theta \in [0, 2\pi),
\]
where \(I_0(\kappa)\) is the modified Bessel function of the first kind of order 0. Then, for any \(n \geq 1\):
\begin{enumerate}
  \item[(a)] There exists an optimal set of \(n\)-means \(Q^* = \{\theta_1^*, \ldots, \theta_n^*\}\) that is symmetric about the mean direction \(\mu\). That is, if we set \(\mu = 0\) without loss of generality, then the optimal codepoints satisfy \(\theta_j^* = -\theta_{n+1-j}^*\) for an appropriate indexing.

  \item[(b)] The Voronoi cells are contiguous arcs \([a_0, a_1], [a_1, a_2], \ldots, [a_{n-1}, a_n]\) with \(a_0 = 0\), \(a_n = 2\pi\), and the boundaries satisfy the symmetry \(a_j = -a_{n-j}\) and
  \begin{equation}\label{eq:boundary_vm}
    |a_j - \theta_j^*| = |a_j - \theta_{j+1}^*|, \quad j = 1, \ldots, n-1.
  \end{equation}

  \item[(c)] For \(\kappa > 0\), cells are smaller near the mode \(\theta = \mu\) and larger in the antipodal direction \(\theta = \mu + \pi\). As \(\kappa \to \infty\), all codepoints converge to \(\mu\) (the degenerate one-point quantizer). As \(\kappa \to 0\), the solution converges to the uniform case: equally spaced codepoints \(\theta_j^* = 2\pi j/n\).

  \item[(d)] The optimal configuration and distortion can be computed numerically as the solution to the following system of \(n\) nonlinear equations derived from Theorem~\ref{thm:main}:
  \begin{align}
    \int_{a_{j-1}}^{a_j} (\theta - \theta_j^*) e^{\kappa \cos(\theta - \mu)} \, d\theta &= 0,
    \quad j = 1, \ldots, n.
    \label{eq:centroid_vm}
  \end{align}
\end{enumerate}

\end{theorem}

\begin{proof}
(a) Assume without loss of generality that \(\mu = 0\) (by rotational invariance). Let \(Q = \{\theta_1, \ldots, \theta_n\}\) be an optimal configuration. Consider its reflection about 0: \(Q' = \{-\theta_1, \ldots, -\theta_n\}\). Since the von Mises density is symmetric (\(h(\theta) = h(-\theta)\)), the distortion for \(Q'\) is equal to that for \(Q\):
\[
V_n(P; Q') = V_n(P; Q).
\]
Now consider the averaged configuration \(\bar{Q} = \frac{1}{2}(Q \cup Q')\), where we take the union and possibly adjust angles to lie in \([0, 2\pi)\). By the convexity of the squared distance function on the circle (within a hemisphere) and Jensen's inequality, we have for any \(\theta\):
\[
\min_{q \in \bar{Q}} |\theta - q|^2 \leq \frac{1}{2} \left( \min_{q \in Q} |\theta - q|^2 + \min_{q \in Q'} |\theta - q|^2 \right).
\]
Multiplying both side by $h(\theta)$ and integrating,

\[
V_n(P; \bar{Q}) \leq \frac{1}{2} \left( V_n(P; Q) + V_n(P; Q') \right) = V_n(P; Q).
\]
Thus, \(\bar{Q}\) is also optimal. Moreover, \(\bar{Q}\) is symmetric by construction: if \(\theta \in \bar{Q}\), then \(-\theta \in \bar{Q}\). Therefore, there exists an optimal symmetric configuration. By relabeling, we can write it as \(\theta_j^* = -\theta_{n+1-j}^*\) for \(j = 1, \ldots, n\).

(b) Theorem~\ref{thm:main} implies that for any optimal configuration, Voronoi cells are contiguous arcs and boundaries are geodesic equidistance points. For the symmetric configuration, symmetry of density and codepoints forces the boundaries to satisfy $a_j=-a_{n-j}$, and the boundary condition reduces to $|a_j-\theta_j^*|=|a_j-\theta_{j+1}^*|$ (shorter-arc distance on the circle, which here coincides with absolute difference since each cell is contained in a semicircle around its representative).

(c) For \(\kappa > 0\), we know that the density \(h(\theta)\) is unimodal with maximum at \(\theta = \mu = 0\). From Theorem~\ref{thm:highres}, we have \(\lambda(\theta) \propto h(\theta)^{1/3} \propto e^{\frac{\kappa}{3} \cos\theta}\). Thus, the optimal point density is higher near \(\theta = 0\), meaning more codepoints and smaller cells in that region.

As \(\kappa \to \infty\), we know that \(h(\theta)\) converges to a Dirac delta at \(\theta = 0\). In this limit, the optimal quantization problem reduces to quantizing a point mass, for which any single codepoint at 0 gives zero distortion. Formally, for any \(\epsilon > 0\), there exists \(\kappa_0\) such that for \(\kappa > \kappa_0\), $P(-\epsilon<\theta<\epsilon)>(1-\epsilon)$.
Thus, as \(\kappa \to \infty\), the optimal codepoints must converge to 0.

As \(\kappa \to 0\), \(h(\theta) \to \frac{1}{2\pi}\), the uniform density. By Theorem~\ref{thm:main}, we say that the optimal codepoints are equally spaced and the Voronoi cells are equal arcs. The symmetry condition then forces \(\theta_j^* = 2\pi j/n\) (with an appropriate shift).

(d) For each cell $[a_{j-1},a_j]$ of an optimal configuration on a great circle, Theorem~\ref{thm:main}(c) gives the local centroid condition
\[
\int_{a_{j-1}}^{a_j} (\theta-\theta_j^*)\,h(\theta)\,d\theta = 0.
\]
Substituting $h(\theta)\propto e^{\kappa\cos(\theta-\mu)}$ and ignoring the common constant factor yields Equation (\ref{eq:centroid_vm}). Then distortion for the von Mises case is
\[
V_n(P)=\frac{1}{2\pi I_0(\kappa)}\sum_{j=1}^n\int_{a_{j-1}}^{a_j} (\theta-\theta_j^*)^2 e^{\kappa\cos(\theta-\mu)}\,d\theta,
\]
which completes the characterization.
\end{proof}

\section{Applications}
\label{sec:appli}
\subsection{Optimal quantization under chordal vs.\ geodesic distortion}
\label{subsec:chordal-vs-geodesic}

In several scenarios it is natural to measure distances on $\mathbb{S}^2$ via the ambient Euclidean norm $d_C(x,y)=\|x-y\|$ rather than the intrinsic geodesic distance $d_G(x,y)=\arccos\langle x,y\rangle$. This subsection compares these two metrics and their impact on optimal quantization along spherical curves.

\paragraph{Relationship between chordal and geodesic distance}
For two points $x, y \in \mathbb{S}^2$ (with $\|x\| = \|y\| = 1$) with geodesic distance $\psi = d_G(x,y)$, the squared chordal and geodesic distances satisfy
\begin{equation*}\label{eq:chordal-geodesic}
d_C^2(x,y)
= \|x-y\|^2
= 2(1-\cos\psi)
= 4\sin^2(\psi/2),\qquad
d_G^2(x,y)=\psi^2.
\end{equation*}
Expanding $\cos\psi$ near $\psi=0$ yields
\begin{equation*}\label{eq:chordal-geodesic-taylor}
d_C^2(x,y)
= \psi^2 - \frac{\psi^4}{12} + O(\psi^6),\qquad
d_G^2(x,y)=\psi^2.
\end{equation*}
The relative error between chordal and geodesic squared distances is
\[
\frac{d_C^2(x,y) - d_G^2(x,y)}{d_G^2(x,y)} = -\frac{\psi^2}{12} + \frac{\psi^4}{360} + O(\psi^6),
\]
which becomes significant for $\psi > 1$ radian (about 57 degrees).
Thus for small separations, chordal and geodesic squared distances coincide to second order; the first nontrivial discrepancy is of order $\psi^4$ and yields a relative error of order $\psi^2$.

On the equator, parameterized by angles $\theta,\alpha\in[0,2\pi)$ with angular separation $\psi=\min\{|\theta-\alpha|,2\pi-|\theta-\alpha|\}$, the squared chordal distance becomes
\[
d_C^2(\theta,\alpha)
= 2-2\cos(\theta-\alpha)
= 4\sin^2\Big(\frac{\theta-\alpha}{2}\Big).
\]

\paragraph{Chordal optimality conditions on the equator}
Let $\Gamma$ be the equator with nonuniform probability $h(\theta)$ and squared chordal distortion. For an optimal $n$–point codebook $\{\theta_j^*\}_{j=1}^n$ and associated Voronoi arcs $[a_{j-1},a_j]$:

\begin{enumerate}[label=\roman*.]
\item \emph{Boundary condition.}
At each internal boundary \(a_j\) the two neighbouring representatives must be equidistant in the chordal metric:
\[
d_C(a_j,\theta_j^*) = d_C(a_j,\theta_{j+1}^*),
\qquad j = 1, \ldots, n-1.
\]
For contiguous arcs shorter than \(\pi\), this is equivalent to
\[
|a_j-\theta_j^*| = |a_j-\theta_{j+1}^*|.
\]

\item \emph{Chordal centroid condition.}
The distortion on a cell is
\[
F_j(\alpha)
=\int_{a_{j-1}}^{a_j} d_C^2(\theta,\alpha)\,h(\theta)\,d\theta
=\int_{a_{j-1}}^{a_j} \bigl(2-2\cos(\theta-\alpha)\bigr) h(\theta)\,d\theta.
\]
The optimal representative \(\theta_j^*\) satisfies \(F_j'(\theta_j^*)=0\), hence
\begin{equation}\label{eq:chordal-centroid}
\int_{a_{j-1}}^{a_j}\sin(\theta-\theta_j^*)\,h(\theta)\,d\theta = 0,
\qquad j = 1,\ldots,n.
\end{equation}
This is the chordal analogue of the intrinsic centroid condition.
\end{enumerate}

For a small cell centred at $\theta_0$ of length $\Delta L$ and density approximately $h(\theta_0)$, therefore,
\[
d_C^2(\theta_0+t,\theta_0)=t^2-\frac{t^4}{12}+O(t^6).
\]
For a cell with representative $\theta_0 + \delta$ (where $\delta$ is a small offset from the center), the distortion to second order in $t$ and $\delta$ is:
\[
d_C^2(\theta_0 + t, \theta_0 + \delta) = (t - \delta)^2 - \frac{(t - \delta)^4}{12} + O((t - \delta)^6).
\]

The contribution to the distortion functional from this cell is:
\begin{align*}
D_{\text{cell}} &= \int_{-\Delta L/2}^{\Delta L/2} d_C^2(\theta_0 + t, \theta_0 + \delta) h(\theta_0 + t) dt \\
&= \int_{-\Delta L/2}^{\Delta L/2} \left[(t - \delta)^2 - \frac{(t - \delta)^4}{12}\right] [h(\theta_0) + h'(\theta_0)t + O(t^2)] dt.
\end{align*}
Expanding and computing term by term, and combining and keeping terms up to $(\Delta L)^3$, and $\delta^2$ survives, but terms like $\delta^2 \Delta L^2$ did not, 
\[
D_{\text{cell}} = h(\theta_0) [\frac{(\Delta L)^3}{12}  +\delta^2 \Delta L] - \frac{h'(\theta_0) \delta (\Delta L)^3}{6} + O((\Delta L)^5).
\]
We also compute the minimum cell distortion and we obtain:
\begin{align*}
D_{\text{cell}}^* 
&= h(\theta_0) \frac{(\Delta L)^3}{12} \left[ 1 - \frac{{h'(\theta_0)^2}(\Delta L)^2}{6h(\theta_0)} + \frac{[h'(\theta_0)]^2 (\Delta L)^2}{12h(\theta_0)^2} \right] + O((\Delta L)^5).
\end{align*}
Now substitute $\Delta L = 1/(n\lambda(\theta))$:
\[
D_{\text{cell}}^* = \frac{h(\theta)}{12n^3 \lambda(\theta)^3} \left[ 1 - \frac{h'(\theta_0)^2}{6n^2 \lambda(\theta)^2h(\theta_0)} + \frac{[h'(\theta_0)]^2}{12h(\theta_0)^2 n^2 \lambda(\theta)^2} \right] + O(n^{-5}).
\]

The total distortion is the integral over all cells (per unit length):
\[
V_n = \int_0^{2\pi} \frac{h(\theta)}{12n^2 \lambda(\theta)^2} \left[ 1 - \frac{h'(\theta_0)^2}{6n^2 \lambda(\theta)^2h(\theta_0)} + \frac{[h'(\theta_0)]^2}{12h(\theta_0)^2 n^2 \lambda(\theta)^2} \right] d\theta.
\]

To leading order:
\[
V_n \sim \frac{1}{12n^2} \int_0^{2\pi} \frac{h(\theta)}{\lambda(\theta)^2} d\theta.
\]

This is identical to the geodesic case at leading order. To see the first correction, we need to include the next term. From the chordal expansion:
\[
d_C^2 = t^2 - \frac{t^4}{12} + \cdots
\]
The $t^4$ term modifies the distortion by a factor. For a uniform distribution, the optimal density remains uniform, but the distortion constant changes. Specifically, for a cell of length $\Delta L$, the average chordal distortion with optimal representative is:
\[
\frac{1}{\Delta L} \int_{-\Delta L/2}^{\Delta L/2} \left(t^2 - \frac{t^4}{12}\right) dt = \frac{(\Delta L)^2}{12} - \frac{(\Delta L)^4}{960}.
\]

Compared to geodesic distortion $\frac{(\Delta L)^2}{12}$, the chordal distortion is smaller by a factor of $1 - \frac{(\Delta L)^2}{80}$. When $\Delta L = 1/(n\lambda)$, this gives a correction factor to the asymptotic distortion.

Thus, for chordal distortion, the high-resolution point density remains $\lambda_C(\theta) \propto h(\theta)^{1/3}$, but the constant in the distortion formula differs.

Figure~\ref{fig:2nd} represents the comparison of intrinsic and chordal squared distances on the sphere as a function of geodesic separation, whereas Figure~\ref{fig:quantizer-comparison} describes the graphical analysis for optimal-7 means for a von Mises density. We also introduce the chordal Lloyd algorithm for nonuniform density, and Figure~\ref{fig:convergence} shows the convergence of this algorithm. Figure~\ref{fig:distortion} represents the convergence analysis for the chordal distortion corresponding to von Mises distribution. Figure~\ref{fig:displacement} shows the convergence of maximum displacement between two consecutive time steps. We have followed the chordal Lloyd algorithm here and as the iteration number increases the graph converges.

\begin{figure}[H]
\centering

\begin{subfigure}[t]{0.48\textwidth}
\centering
\begin{tikzpicture}[scale=0.8]
\begin{axis}[
    xlabel={$\psi$ (radians)},
    ylabel={Squared distance},
    legend pos=north west,
    grid=both,
    domain=0:3.14159,
    samples=150
]
\addplot[blue, thick] {x^2};
\addplot[red, thick] {2 - 2*cos(deg(x))};
\addplot[green, thick, dashed] {x^2 - x^4/12};

\legend{
  {$d_G^2(\psi)=\psi^2$},
  {$d_C^2(\psi)=2-2\cos\psi$},
  {$d_C^2$ (4th-order approx.)}
}
\end{axis}
\end{tikzpicture}
\caption{Squared geodesic and chordal distances as functions of \(\psi\), together with the fourth-order approximation of \(d_C^2\). For small \(\psi\), all curves coincide up to quadratic order; the deviation of \(d_C^2\) from \(\psi^2\) becomes visible at larger angles.}
\end{subfigure}
\hfill

\begin{subfigure}{0.48\textwidth}
\centering
\begin{tikzpicture}[scale=0.8]
\begin{axis}[
    xlabel={$\psi$ (radians)},
    ylabel={Relative error},
    legend pos=north east,
    grid=both,
    domain=0:3.14,
    samples=50
]
\addplot[red, thick] {(2 - 2*cos(deg(x)) - x^2)/x^2};
\legend{$(d_C^2 - d_G^2)/d_G^2$};
\end{axis}
\end{tikzpicture}
\caption{Relative error between chordal and geodesic squared distances. The error grows quadratically in $\psi$ for small angles and becomes significant beyond about one radian.}
\end{subfigure}
\caption{Comparison of intrinsic and chordal squared distances on $\mathbb{S}^2$ as functions of geodesic separation $\psi$.}
\label{fig:2nd}
\end{figure}
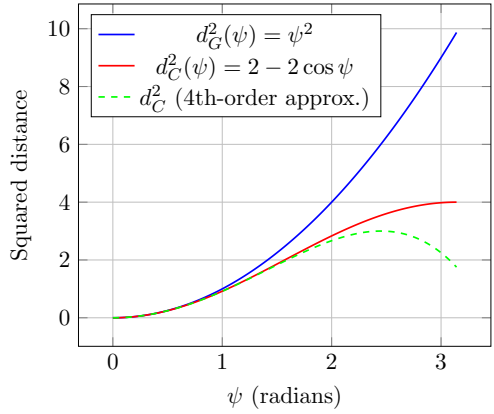
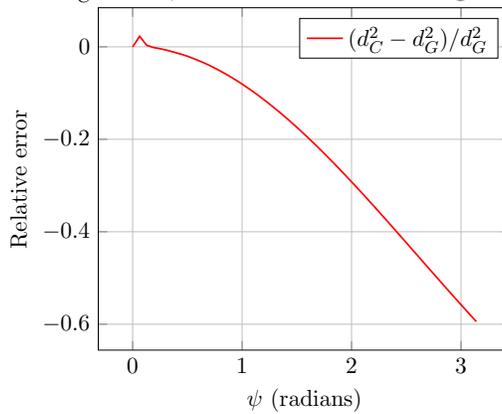

\begin{figure}[htbp]
\centering
\begin{tikzpicture}
\begin{axis}[
    width=1.0\linewidth,
    height=6cm,
    xlabel={$\theta$ (radians)},
    ylabel={Optimal codepoints and density},
    grid=both,
    domain=0:2*pi,
    samples=200,
    xmin=0, xmax=6.28,
    ymin=0, ymax=1.2,
    xtick={0,1.57,3.14,4.71,6.28},
    xticklabels={0,$\pi/2$,$\pi$,$3\pi/2$,$2\pi$},
    ytick={0,0.2,0.4,0.6,0.8,1.0},
    legend pos=north east,
    legend cell align={left}
]
\addplot[gray!30, fill=gray!30, domain=0:2*pi]
  {0.03258126*exp(3*cos(deg(x)))} \closedcycle;

\addplot[black, thick, domain=0:2*pi]
  {0.03258126*exp(3*cos(deg(x)))};

\addplot[blue, only marks, mark=*, mark size=3, mark options={solid}] coordinates {
    (0.365,0.05) (0.784,0.05) (1.387,0.05) (3.142,0.05)
    (4.896,0.05) (5.499,0.05) (5.918,0.05) 
};


\addplot[red, only marks, mark=x, mark size=3, mark options={solid}] coordinates {
    (0.363,0.1) (0.781,0.1) (1.384,0.1) (3.142,0.1)
    (4.900,0.1) (5.502,0.1) (5.921,0.1) 
};

\foreach \x/\y in {0.392/0.387, 1.178/1.171, 1.963/1.956, 2.749/2.742,
                   3.534/3.527, 4.320/4.313, 5.105/5.098, 5.891/5.884} {
    \edef\temp{\noexpand\draw[->, black!50, thin]
        (axis cs:\x,0.06) -- (axis cs:\y,0.09);}
    \temp
}

\legend{
        Chordal quantizers (red crosses),
        Density $h(\theta)$ (von Mises $\kappa=3$),
        Geodesic quantizers (blue dots)}

\end{axis}
\end{tikzpicture}
\caption{Optimal 7-means for von Mises distribution. Chordal quantizers (red crosses) shift slightly toward $\pi$ compared to geodesic quantizers (blue dots).}
\label{fig:quantizer-comparison}
\end{figure}
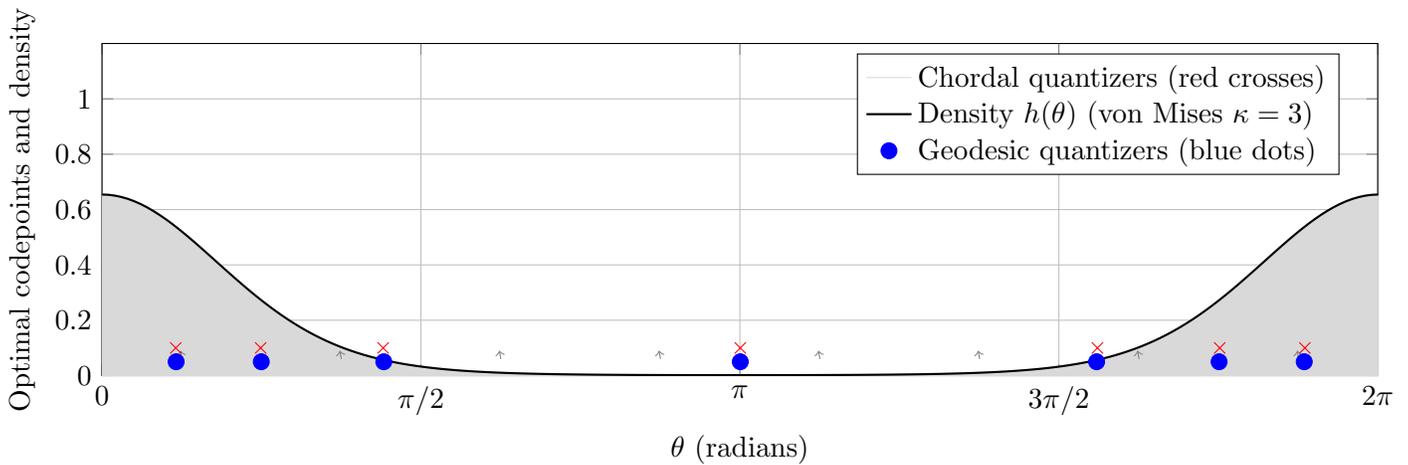

\begin{table}[H]
\centering
\caption{Numerical comparison of optimal 7-means for von Mises distribution ($\kappa=3, \mu=0$)}
\begin{tabular}{lccccccc}
\hline
Metric & $\theta_1^*$ & $\theta_2^*$ & $\theta_3^*$ & $\theta_4^*$ & $\theta_5^*$ & $\theta_6^*$ & $\theta_7^*$ \\
\hline
Geodesic & 0.365 & 0.784 & 1.387 & 3.142 & 4.896 & 5.499 & 5.918\\
Chordal & 0.363 & 0.781 & 1.384 & 3.142 & 4.900 & 5.502 & 5.921 \\
Difference & 0.002 & 0.003 & 0.003 & 0.000 & 0.004 & 0.003 & 0.003 \\
\hline
\end{tabular}
\caption*{Chordal quantizers are shifted slightly toward the antipodal point ($\theta=\pi$).}
\label{tab:chordal-geodesic-comparison}
\end{table}

\begin{algorithm}[htbp]
\caption{Chordal Lloyd algorithm for nonuniform density}
\label{alg:chordal-lloyd}
\begin{algorithmic}[1]
\Require Density $h(\theta)$, initial codebook $\Theta^{(0)}=\{\theta_1^{(0)},\ldots,\theta_n^{(0)}\}$, tolerance $\epsilon>0$, max iterations $M$
\Ensure Optimized codebook $\Theta^*$
\State $t \gets 0$
\Repeat
  \State \textbf{Assignment:} Partition $[0,2\pi]$ into cells $R_j^{(t)}$ using chordal distance.
  \Statex \hspace{\algorithmicindent} $R_j^{(t)}=\{\theta:\ d_C(\theta,\theta_j^{(t)})\le d_C(\theta,\theta_k^{(t)}),\ \forall k\ne j\}$
  \Statex \hspace{\algorithmicindent} where $d_C(\theta,\alpha)=2|\frac{\sin(\theta-\alpha)}{2}|$.
  \State \textbf{Update:} For each $j=1,\ldots,n$, compute $\theta_j^{(t+1)}$ by solving
  \Statex \hspace{\algorithmicindent} $\int_{R_j^{(t)}} \sin(\theta-\theta_j^{(t+1)})\,h(\theta)\,d\theta = 0$
  \Statex \hspace{\algorithmicindent} using Newton's method initialized at $\theta_j^{(t)}$.
  \State \textbf{Check Convergence:} $\Delta^{(t)} = \max_j \left|\theta_j^{(t+1)}-\theta_j^{(t)}\right|$
  \State $t \gets t+1$
\Until{$\Delta^{(t-1)}<\epsilon$ \textbf{or} $t\ge M$}
\State $\Theta^* \gets \Theta^{(t)}$.
\end{algorithmic}
\end{algorithm}

\begin{figure}[htbp]
\centering
\includegraphics[width=0.9\linewidth]{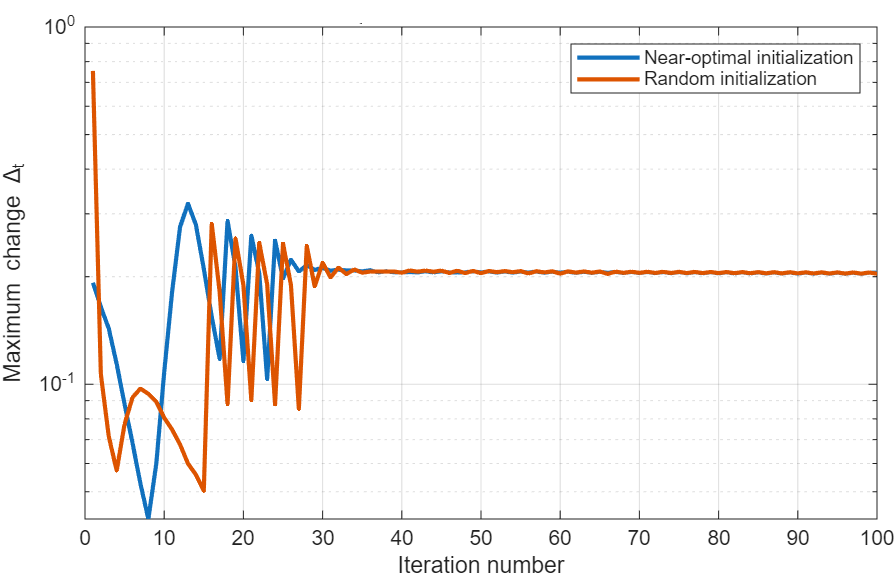}
\caption{Convergence of Algorithm 1 for n=7, $\kappa$=3}
\label{fig:convergence}
\end{figure}
\begin{figure}[htbp]
\centering
\includegraphics[width=0.9\linewidth]{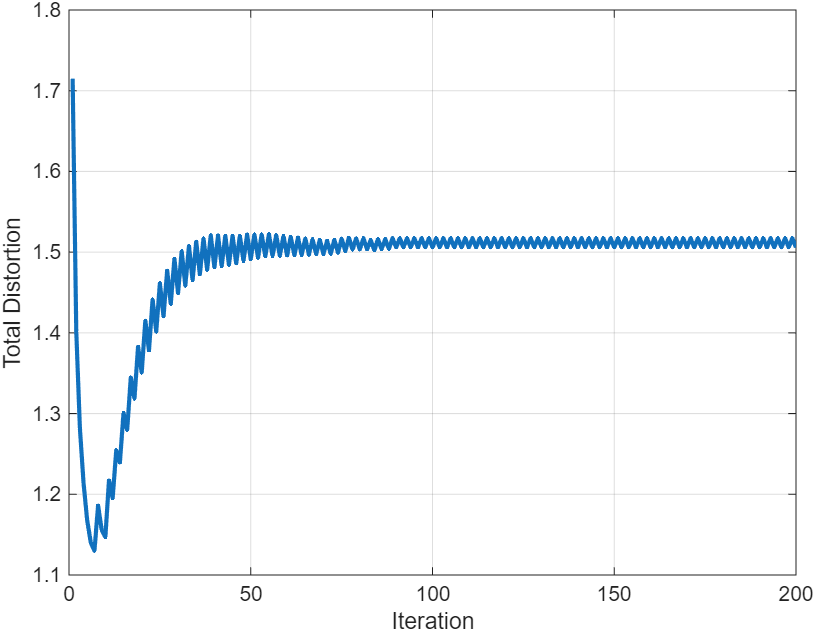}
\caption{Analysis and convergence for the distortion error corresponding to chordal distribution}
\label{fig:distortion}
\end{figure}
\begin{figure}[htbp]
\centering
\includegraphics[width=0.8\linewidth]{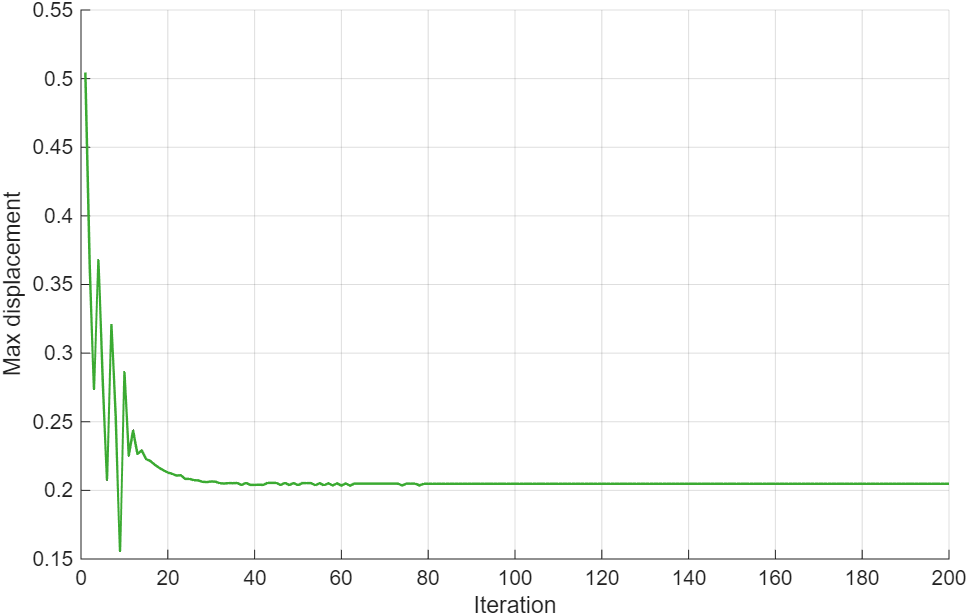}
\caption{Maximum displacement analysis and convergence due to Lloyd iteration scheme.}
\label{fig:displacement}
\end{figure}
Figure~\ref{fig:convergence} shows the convergence behavior of Algorithm~\ref{alg:chordal-lloyd} by plotting the maximum update versus iteration $t$ on a logarithmic scale for 
$n=7$ and a von Mises density with $\kappa=3$. When initialized near the optimal configuration, the 
algorithm rapidly enters a Newton-like regime and exhibits fast, nearly quadratic convergence, 
reaching machine precision within a few iterations. In contrast, a random initial guess leads to a 
slower initial decay due to repeated changes in the Voronoi partition, which introduce nonsmoothness 
into the optimization landscape. Once the partition stabilizes, convergence accelerates, 
demonstrating the robustness of the algorithm with respect to initialization.
\subsection{Quantization of mixture distributions on spherical curves}
\label{subsec:mixture}

Mixture distributions on spherical curves arise naturally in applications where data originates from multiple sources or populations. Consider a probability measure $P$ on a great circle $\Gamma$ (parameterized by $\theta \in [0,2\pi)$) with density given by a mixture of $K$ von Mises distributions:

\begin{equation*}
h(\theta) = \sum_{k=1}^{K} w_k h_k(\theta), \quad h_k(\theta) = \frac{1}{2\pi I_0(\kappa_k)} e^{\kappa_k \cos(\theta - \mu_k)},
\label{eq:mixture}
\end{equation*}
with weights $w_k>0$, $\sum_k w_k=1$, mean directions $\mu_k$, and concentrations $\kappa_k\ge0$. For squared geodesic distortion, the general theory yields:
\begin{enumerate}[label=\roman*.]
\item By Theorem~\ref{thm:main}(a), Voronoi cells $[a_{j-1},a_j]$ are contiguous arcs with
$a_0=0$, $a_n=2\pi$, and
\[
|a_j-\theta_j^*|=|a_j-\theta_{j+1}^*|,\qquad j=1,\ldots,n-1.
\]

\item By Theorem~\ref{thm:main}(c), on each cell the representative $\theta_j^*$ minimizes
\[
F_j(\alpha)=\int_{a_{j-1}}^{a_j}|\theta-\alpha|^2\, h(\theta)\,d\theta
=\sum_{k=1}^{K}w_k\int_{a_{j-1}}^{a_j}|\theta-\alpha|^2\, h_k(\theta)\,d\theta,
\]
so differentiating yields
\[
\sum_{k=1}^{K}w_k\int_{a_{j-1}}^{a_j}(\theta-\theta_j^*)\,e^{\kappa_k\cos(\theta-\mu_k)}\,d\theta=0,
\qquad j=1,\ldots,n.
\]

\item If a component has large concentration $\kappa_k$ while others remain moderate,
then the mass near $\mu_k$ dominates locally and, in the high--resolution regime,
approximately $n_k\approx n w_k$ codepoints are allocated in a neighborhood of $\mu_k$,
with tighter spacing for larger $\kappa_k$.
We propose an algorithm for mixture quantization for nonuniform densities; see
Algorithm~\ref{alg:mixture-quant}.
\end{enumerate}

\begin{algorithm}[htbp]
\caption{Mixture quantization for von Mises distributions}
\label{alg:mixture-quant}
\begin{algorithmic}[1]
\Require Mixture weights $w_k$, parameters $(\mu_k,\kappa_k)$ for $k=1,\ldots,K$, initial codebook $Q^{(0)}=\{q_1^{(0)},\ldots,q_n^{(0)}\}$,
tolerance $\epsilon>0$, discretization points $\{\theta_1,\ldots,\theta_M\}$.
\Ensure Optimized codebook $Q^*=\{q_1^*,\ldots,q_n^*\}$.

\State $t \gets 0$
\Repeat
  \State \textbf{E-step:} Compute responsibilities for each $\theta_j$ and component $k$.
  \Statex \hspace{\algorithmicindent} $\displaystyle
  \gamma_{jk}^{(t)}=\frac{w_k\,h_k(\theta_j)}{\sum_{l=1}^K w_l\,h_l(\theta_j)},
  \qquad
  h_k(\theta)=\frac{1}{2\pi I_0(\kappa_k)}e^{\kappa_k\cos(\theta-\mu_k)}.$

  \State \textbf{Assignment:} Assign each $\theta_j$ to the closest codepoint.
  \Statex \hspace{\algorithmicindent} $\displaystyle
  \ell_j^{(t)}=\arg\min_{i=1,\ldots,n}\ \sum_{k=1}^K \gamma_{jk}^{(t)}\, d_G(\theta_j,q_i^{(t)})^2,$
  \Statex \hspace{\algorithmicindent} where $\displaystyle d_G(\theta,q)=\min\{|\theta-q|,\ 2\pi-|\theta-q|\}.$

  \State \textbf{M-step:} Update each codepoint as the responsibility-weighted mean.
  \Statex \hspace{\algorithmicindent} $\displaystyle
  q_i^{(t+1)}=
  \frac{\sum_{j:\,\ell_j^{(t)}=i}\Big(\sum_{k=1}^K \gamma_{jk}^{(t)}\Big)\,\theta_j}
       {\sum_{j:\,\ell_j^{(t)}=i}\Big(\sum_{k=1}^K \gamma_{jk}^{(t)}\Big)},
  \qquad i=1,\ldots,n.$
  \Statex \hspace{\algorithmicindent} If the denominator is $0$, set $q_i^{(t+1)}\gets q_i^{(t)}$.

  \State Normalize to $[0, 2\pi)$: $q_i^{(t+1)} \gets q_i^{(t+1)} \bmod 2\pi$ for all $i$.
  \State Sort codepoints: $\{q_1^{(t+1)},\ldots,q_n^{(t+1)}\} \gets$ sort in ascending order.
  \State $\Delta \gets \max_{i=1,\ldots,n} d_G\!\big(q_i^{(t+1)},q_i^{(t)}\big)$
  \State $t \gets t+1$
\Until{$\Delta<\epsilon$}

\State \Return $Q^*=\{q_1^{(t)},\ldots,q_n^{(t)}\}$
\end{algorithmic}
\end{algorithm}

\begin{table}[!htbp]
\centering
\caption{Optimal 4-means for a mixture of two von Mises distributions}
\label{tab:mixture-results}
\begin{tabular}{ccccccc}
\toprule
$w_1$ & $\kappa_1$ & $\kappa_2$ & $\theta_1^*$ & $\theta_2^*$ & $\theta_3^*$ & $\theta_4^*$ \\
\midrule
0.5 & 2  & 2 & 0.79 & 2.36 & 3.92 & 5.48 \\
0.7 & 5  & 1 & 0.79 & 2.36 & 3.92 & 5.48 \\
0.3 & 10 & 2 & 0.78 & 2.35 & 3.92 & 5.49 \\
\bottomrule
\end{tabular}
\end{table}

\subsection{Cosine–modulated density on a great circle}
We consider the equator $\Gamma$ identified with $\theta\in[0,2\pi)$ and a simple nonuniform density \cite{MardiaJupp2000}
\begin{equation*}\label{eq:cos-density}
h_\alpha(\theta)
:= \frac{1+\alpha\cos\theta}{2\pi},\qquad \theta\in[0,2\pi),\quad |\alpha|<1,
\end{equation*}
which models a single preferred direction at $\theta=0$ with strength controlled by $\alpha$. We have used $\alpha$ as a suffix here to keep uniformity in this problem. The normalization condition
\[
\int_0^{2\pi} h_\alpha(\theta)\,d\theta
= \frac{1}{2\pi}\int_0^{2\pi}\big(1+\alpha\cos\theta\big)d\theta
= 1
\]
is immediate, and positivity follows from $1+\alpha\cos\theta\ge 1-|\alpha|>0$.
The high–resolution optimal point density is
\begin{equation*}\label{eq:cos-lambda-def}
\lambda_\alpha(\theta)
= \frac{h_\alpha(\theta)^{1/3}}{\displaystyle\int_0^{2\pi} h_\alpha(u)^{1/3}\,du},\qquad \theta\in[0,2\pi).
\end{equation*}
Therefore, the denominator $Z_\alpha
:= \int_0^{2\pi} h_\alpha(\theta)^{1/3}\,d\theta
= (2\pi)^{2/3}\int_0^{2\pi}\big(1+\alpha\cos\theta\big)^{1/3}\,\frac{d\theta}{2\pi}.$

For small $|\alpha|$ one may expand $(1+\alpha\cos\theta)^{1/3}$ as
\[
(1+\alpha\cos\theta)^{1/3}
= 1+\frac{1}{3}\alpha\cos\theta - \frac{\alpha^2}{9}\cos^2\theta + O(\alpha^3)
= 1+\frac{\alpha}{3}\cos\theta - \frac{\alpha^2}{9}\cos^2\theta + O(\alpha^3).
\]
Therefore,
\begin{equation*}\label{eq:cos-Z-expansion}
Z_\alpha
= (2\pi)^{2/3}\left(1 - \frac{\alpha^2}{18} + O(\alpha^3)\right).
\end{equation*}
Hence, we obtain
\begin{equation*}\label{eq:cos-lambda-expansion}
\lambda_\alpha(\theta)
= \frac{\big(1+\alpha\cos\theta\big)^{1/3}}{2\pi\,\big(1-\frac{\alpha^2}{18}+O(\alpha^3)\big)}
= \frac{1}{2\pi}\Big(1+\tfrac{\alpha}{3}\cos\theta + O(\alpha^2)\Big),
\end{equation*}
showing that the point density is increased near $\theta=0$ when $\alpha>0$ and reduced near $\theta=\pi$. This is the quantification of the asymptotic point density.

Now, $\{R_{j,n}\}_{j=1}^n$ being the Voronoi cells of an optimal $n$–point quantizer, with $R_{j,n}$ containing a point $s_{j,n}$. Theorem~\ref{thm:highres} implies
\begin{equation*}\label{eq:cos-cell-length}
|R_{j,n}|
\sim \frac{1}{n\,\lambda_\alpha(s_{j,n})}
\propto n^{-1}\big(h_\alpha(s_{j,n})\big)^{-1/3},
\qquad n\to\infty.
\end{equation*}
Thus cells are shortest where $h_\alpha$ is largest (near $\theta=0$ for $\alpha>0$) and longest where $h_\alpha$ is smallest (near $\theta=\pi$).

For the asymptotic error, Theorem~\ref{thm:asymp+vm} yields
\begin{equation*}\label{eq:cos-asymp-error}
\lim_{n\to\infty} n^2 V_n(P_\alpha)
= \frac{1}{12}Z_\alpha^3,
\end{equation*}

Therefore,
\begin{equation}\label{eq:cos-asymp-error-expansion}
\lim_{n\to\infty} n^2 V_n(P_\alpha)
= \frac{(2\pi)^2}{12}\left(1 - \frac{\alpha^2}{6} + O(\alpha^3)\right).
\end{equation}
The leading term $(2\pi)^2/12$ is the uniform–density constant; the negative $O(\alpha^2)$ correction reflects the reduction in distortion due to concentrating mass (and codepoints) near the preferred direction. This corresponds to the asymptotic cell lengths and its associated error \cite{GershoGray1991}.

\subsection{Bimodal stationary phase law on the great circle}
We consider a bimodal phase distribution on the great circle, motivated by stationary laws of phase dynamics with two preferred orientations \cite{MardiaJupp2000}. We consider
\begin{equation}\label{eq:bimodal-density}
h_\beta(\theta)
:= \frac{1}{2\pi I_0(\beta)}\exp\big(\beta\cos(2\theta)\big),\qquad \beta>0,\ \theta\in[0,2\pi),
\end{equation}
where $I_0(\beta)$ is the modified Bessel function of order zero. The factor $\cos(2\theta)$ produces two modes per $2\pi$--period, located at $\theta=0$ and $\theta=\pi$. Normalization is standard:
\[
\int_0^{2\pi} h_\beta(\theta)\,d\theta
= \frac{1}{2\pi I_0(\beta)}\int_0^{2\pi} e^{\beta\cos(2\theta)}\,d\theta
=1.
\]

By Theorem~\ref{thm:highres}, the high--resolution point density is
\begin{equation}\label{eq:bimodal-lambda}
\lambda_\beta(\theta)
= \frac{h_\beta(\theta)^{1/3}}{\displaystyle\int_0^{2\pi} h_\beta(u)^{1/3}\,du}
= \frac{\exp\big(\frac{\beta}{3}\cos(2\theta)\big)}{\displaystyle\int_0^{2\pi}\exp\big(\tfrac{\beta}{3}\cos(2u)\big)\,du}.
\end{equation}
From the denominator, we obtain
\begin{equation}\label{eq:bimodal-denom}
\int_0^{2\pi}\exp\big(\tfrac{\beta}{3}\cos(2u)\big)\,du
=2\pi I_0\!\left(\frac{\beta}{3}\right).
\end{equation}

We also use the expansion
\[
\exp\!\left(\frac{\beta}{3}\cos(2\theta)\right)
=1+\frac{\beta}{3}\cos(2\theta)+O(\beta^2).
\]
Therefore,
\begin{equation}\label{eq:bimodal-lambda-expansion}
\lambda_\beta(\theta)
=\frac{1}{2\pi I_0\!\left(\frac{\beta}{3}\right)}
\left[1+\frac{\beta}{3}\cos(2\theta)+O(\beta^2)\right].
\end{equation}

The asymptotic error is given by Theorem~\ref{thm:asymp+vm} as
\begin{equation}\label{eq:bimodal-asymp-error}
\lim_{n\to\infty} n^2 V_n(P_\beta)
= \frac{1}{12}\left(\int_0^{2\pi} h_\beta(\theta)^{1/3}\,d\theta\right)^3.
\end{equation}
We can rewrite the integral using the periodicity of $\cos(2\theta)$:
\[
\int_0^{2\pi} h_\beta(\theta)^{1/3}\,d\theta
= \frac{1}{(2\pi I_0(\beta))^{1/3}}\int_0^{2\pi}\exp\left(\tfrac{\beta}{3}\cos(2\theta)\right)\,d\theta
= \frac{2\pi}{(2\pi I_0(\beta))^{1/3}}\,I_0\!\left(\tfrac{\beta}{3}\right).
\]
Therefore,
\begin{equation}\label{eq:bimodal-Z}
\int_0^{2\pi} h_\beta(\theta)^{1/3}\,d\theta
= (2\pi)^{2/3}\frac{I_0(\beta/3)}{I_0(\beta)^{1/3}},
\end{equation}
and so
\begin{equation}\label{eq:bimodal-asymp-error-explicit}
\lim_{n\to\infty} n^2 V_n(P_\beta)
= \frac{1}{12}(2\pi)^2\left(\frac{I_0(\beta/3)}{I_0(\beta)^{1/3}}\right)^3.
\end{equation}
This expression shows how the asymptotic constant depends on $\beta$ through the ratio of Bessel functions. As $\beta\to0$, $I_0(\beta)\to1$ and $I_0(\beta/3)\to1$, recovering the uniform constant $(2\pi)^2/12$. For large $\beta$, standard asymptotics for $I_0$ can be used to obtain the leading dependence on $\beta$.

\subsection{Nonuniform quadrature on a spherical arc}
We now show how high--resolution quantizers can be used to construct quadrature rules adapted to a nonuniform density on a spherical curve. Let $\Gamma$ be a geodesic arc of length $L$, parameterized by arc--length $s\in[0,L]$, and let $P$ have a continuous nonuniform density $h(s)$ with respect to $ds$, normalized so that $\int_0^{L}h(s)\,ds=1$ \cite{Pag`esPrintems2003}. Here we adapt the same idea to intrinsic quantization on spherical curves.

Suppose we wish to approximate integrals of the form
\begin{equation}\label{eq:quad-target}
I_j=\int_\Gamma f(s)\,h(s)\,ds
\end{equation}
by a discrete rule
\begin{equation}\label{eq:quad-rule}
Q_n(f):=\sum_{j=1}^n w_j f(q_j),
\end{equation}
where $q_j\in\Gamma$ and $w_j>0$ are weights summing to $1$. A natural choice is to take $q_j$ as the optimal quantizers for $P$ and $w_j$ as the masses of their Voronoi cells.

Let $Q_n^*=\{q_{1,n}^*,\dots,q_{n,n}^*\}$ be an optimal set of $n$--means with Voronoi cells $R_{j,n}$ and cell lengths $\Delta L_{j,n}=|R_{j,n}|$. Define
\[
w_{j,n}:=P(R_{j,n})=\int_{R_{j,n}}h(s)\,ds.
\]
Then the quadrature rule
\begin{equation}\label{eq:quad-opt}
Q_n(f)=\sum_{j=1}^n w_{j,n}f(q_{j,n}^*)
\end{equation}
is adapted to the distribution $P$. For smooth $f$ and small cells, one can Taylor expand $f$ around $q_{j,n}^*$:
\[
f(s)=f(q_{j,n}^*) + f'(q_{j,n}^*)(s-q_{j,n}^*) + \tfrac{1}{2}f''(q_{j,n}^*)(s-q_{j,n}^*)^2 + \cdots.
\]
Substituting into \(I_j\) gives
\begin{align*}
I_j&=
f(q_{j,n}^*)\underbrace{\int_{R_{j,n}}h(s)\,ds}_{=w_{j,n}}
+f'(q_{j,n}^*)\underbrace{\int_{R_{j,n}}(s-q_{j,n}^*)h(s)\,ds}_{=:A_j}
\\
&\quad+\frac12 f''(q_{j,n}^*)\underbrace{\int_{R_{j,n}}(s-q_{j,n}^*)^2h(s)\,ds}_{=:B_j}
+O\!\left(\int_{R_{j,n}}|s-q_{j,n}^*|^3h(s)\,ds\right).
\end{align*}

We define the local error as
\[
E_j=I_j-w_{j,n}f(q_{j,n}^*).
\]
Thus
\[
E_j = f'(q_{j,n}^*)A_j+\frac12 f''(q_{j,n}^*)B_j+O\!\left(\int_{R_{j,n}}|s-q_{j,n}^*|^3h(s)\,ds\right).
\]
Since \(q_{j,n}^*\) minimizes the local distortion functional
\[
F_j(q)=\int_{R_{j,n}}d_G(s,q)^2\,h(s)\,ds,
\]
and on a geodesic arc parameterized by arc--length we have \(d_G(s,q)=|s-q|\), it follows that
\[
F_j(q)=\int_{s_{j-1}}^{s_j}(s-q)^2h(s)\,ds,\qquad q\in\mathbb{R}.
\]
The minimizer satisfies
\[
\frac{dF_j}{dq}\Big|_{q=q_{j,n}^*}= -2\int_{s_{j-1}}^{s_j}(s-q_{j,n}^*)h(s)\,ds =0,
\]
which is exactly
\[
A_j=\int_{R_{j,n}}(s-q_{j,n}^*)h(s)\,ds=0.
\]

From Theorem~\ref{thm:main}, the first--order term cancels in the approximation error on each cell, and the leading contribution comes from the second derivative and the local second moment of $h$ on the cell.

For the asymptotic analysis, Theorem~\ref{thm:highres} yields the limiting point density
\begin{equation}\label{eq:quad-lambda}
\lambda(s)=\frac{h(s)^{1/3}}{\displaystyle\int_0^{L}h(u)^{1/3}\,du},
\end{equation}
such that asymptotically
\begin{equation}\label{eq:quad-cell-length}
\Delta L_{j,n}\sim \frac{1}{n\,\lambda(s_{j,n})}
\propto \frac{1}{n\,h(s_{j,n})^{1/3}},
\qquad s_{j,n}\in R_{j,n}.
\end{equation}
Moreover,
\[
w_{j,n}
= \int_{R_{j,n}}h(s)\,ds
\approx h(s_{j,n})\,\Delta L_{j,n}.
\]
Therefore,
\begin{equation}\label{eq:quad-weights-asymp}
w_{j,n}
\sim \frac{h(s_{j,n})}{n\,\lambda(s_{j,n})}
= \frac{h(s_{j,n})}{n}\cdot\frac{\displaystyle\int_0^{L}h(u)^{1/3}\,du}{h(s_{j,n})^{1/3}}
= \frac{\Big(\int_0^L h(u)^{1/3}\,du\Big)}{n}\,h(s_{j,n})^{2/3}.
\end{equation}
Thus, up to the common normalization factor $C=\int_0^L h(u)^{1/3}\,du$, the asymptotic quadrature weights behave as
\begin{equation}\label{eq:quad-weights-normalized}
w_{j,n}\propto h(s_{j,n})^{2/3},\qquad \sum_{j=1}^n w_{j,n}=1.
\end{equation}
Therefore, in the high--resolution limit, the quadrature nodes concentrate according to $h^{1/3}$, while the weights scale like $h^{2/3}$. In particular, for the uniform case $h\equiv 1/L$, one recovers equally spaced nodes and equal weights $w_{j,n}=1/n$, which correspond to the classical trapezoidal rule on an interval (or circle).

This construction provides an $L^2$--adapted quadrature rule on $\Gamma$ that is connected to the density $h$. It can be viewed as a manifold analogue of quantization--based integration schemes developed for Euclidean measures.

\section{Conclusions}
\label{sec:conclu}
In our article, we presented an asymptotic and algorithmic analysis for optimal quantization for nonuniform densities on spherical curves. We showed the intrinsic geodesic distortion and chordal distribution by the previously established centroid criteria and high-resolution asymptotic analysis. The distortion errors corresponding to both geodesic and chordal distances are also quantified, and we showed that they agreed upto order 2. We established optimality conditions for continuous, strictly positive nonuniform densities on geodesic curves, in the form of an intrinsic centroid condition and a geodesic boundary condition that characterizes an optimal Voronoi partition. By extending the standard result provided by Gersho and others  \cite{GershoGray1991}, we derived the point-density law $\lambda(s)\propto h(s)^{1/3}$ and the sharp asymptotic error formula. Theorem~\ref{thm:main} was applied to the von
Mises distribution to analyze including symmetry properties and high-resolution point densities. 
Finally, we used our results in several examples: (i) Optimal quantization analysis under chordal vs. geodesic distortion. (ii) Quantization analysis of mixture distributions on spherical curves. (iii) Explicit high-resolution quantizers and error constants for a cosine-modulated density on a great circle. (iv) Optimal quantizers for bimodal von Mises–type distributions, which shows how the optimal codepoints split between multiple modes according to the $h^{1/3}$ principle; and (v) An intrinsic, nonuniform quadrature rule on spherical arcs whose nodes and weights are generated from optimal quantizers, extending Euclidean/functional quantization-based integration schemes to the spherical setting \citep{Pag`esPrintems2003}.

Moreover, there can be several future possibilities based on our work. Particularly, instead of spherical curves, if we can consider some arbitrary geometries, such as triangular surfaces, polygonal surfaces, how the nonuniform density plays a role could be a wise direction for future endeavours.

\end{document}